\documentclass[sn-mathphys,Numbered]{sn-jnl}% Math and Physical Sciences Reference Style
\usepackage{graphicx}%
\usepackage{multirow}%
\usepackage{amsmath,amssymb,amsfonts}%
\usepackage{amsthm}%
\usepackage{mathrsfs}%
\usepackage[title]{appendix}%
\usepackage{xcolor}%
\usepackage{textcomp}%
\usepackage{manyfoot}%
\usepackage{booktabs}%
\usepackage{algorithm}%
\usepackage{algorithmicx}%
\usepackage{algpseudocode}%
\usepackage{listings}%
\usepackage{tikz}

\theoremstyle{thmstyleone}%
\newtheorem{theorem}{Theorem}%  meant for continuous numbers
\theoremstyle{thmstyletwo}%
\newtheorem{lemma}{Lemma}

\theoremstyle{thmstylethree}%

\begin{document}

\title[Bounds on tree distribution in number theory]{Bounds on tree distribution in number theory}
%
% bounds on trees distribution in number theory.
%
%%=============================================================%%
%% Prefix	-> \pfx{Dr}
%% GivenName	-> \fnm{Joergen W.}
%% Particle	-> \spfx{van der} -> surname prefix
%% FamilyName	-> \sur{Ploeg}
%% Suffix	-> \sfx{IV}
%% NatureName	-> \tanm{Poet Laureate} -> Title after name
%% Degrees	-> \dgr{MSc, PhD}
%% \author*[1,2]{\pfx{Dr} \fnm{Joergen W.} \spfx{van der} \sur{Ploeg} \sfx{IV} \tanm{Poet Laureate} 
%%                 \dgr{MSc, PhD}}\email{iauthor@gmail.com}
%%=====================================================================================%%

\author[1]{\fnm{Roberto} \sur{Conti}}\email{roberto.conti@sbai.uniroma1.it}
\equalcont{These authors contributed equally to this work.}

\author[2]{\fnm{Pierluigi} \sur{Contucci}}\email{pierluigi.contucci@unibo.it}
\equalcont{These authors contributed equally to this work.}

\author*[3]{\fnm{Vitalii} \sur{Iudelevich}}\email{vitaliiiudelevich@gmail.com}
\equalcont{These authors contributed equally to this work.}

\affil[1]{\orgdiv{Dipartimento di Scienze di Base e Applicate per l'Ingegneria}, \orgname{Sapienza Universit\`a di Roma}, \orgaddress{\street{via A. Scarpa 16}, \city{Rome}, \postcode{00161}, 
% \state{State}, 
\country{Italy}}}

\affil*[2]{\orgdiv{Department of Mathematics}, \orgname{Alma Mater Studiorum - University of Bologna}, \orgaddress{\street{Piazza di Porta San Donato, 5}, \city{Bologna}, \postcode{40126}, \country{Italy}}}

\affil[3]{\orgdiv{Faculty of Mechanics and Mathematics}, \orgname{Lomonosov Moscow State University}, \orgaddress{\street{ Leninskie Gory, 1}, \city{Moscow}, \postcode{119991}, \country{Russia}}}

%%==================================%%
%% sample for unstructured abstract %%
%%==================================%%

\abstract{The occurrence and the distribution of patterns of trees associated to natural numbers are investigated. Bounds from above and below are proven for certain natural quantities.}


%%\pacs[JEL Classification]{D8, H51}

%%\pacs[MSC Classification]{35A01, 65L10, 65L12, 65L20, 65L70}

\maketitle

\section{Introduction}\label{sec1}

%For unexplained and/or more detailed notions and terminology, see the preliminaries.

The fundamental theorem of arithmetic implies that it is always possible to associate to a  natural number a unique % decorated 
labeled rooted tree that keeps track of the prime decomposition of the number and, iteratively, of the prime decompositions of all the involved exponents, as illustrated in \cite{DG,Iud21,Iud22,CC,DKV}. Sometimes this structure has been referred to as {\it prime tower factorization}.
% XXXX ADD EXAMPLES WITH FIGURES
For instance, the decomposition of $320 = 2^6 \cdot 5= 2^{2 \cdot 3} \cdot 5$ leads to the labeled rooted tree depicted in Fig. 1 where all the labels are prime numbers.
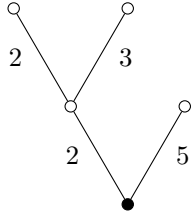
\begin{figure}
    \centering
    \begin{tikzpicture}
\tikzstyle{hollow node}=[circle,draw,inner sep=1.5]
\tikzstyle{solid node}=[circle,draw,inner sep=1.5,fill=black]
\node(0)[solid node]{}
child[grow=60]{node[hollow node]{} edge from parent node[xshift=10]{$5$}}
child[grow=120]{node[hollow node]{}  
child[grow=60]{node[hollow node]{} edge from parent node[xshift=10]{$3$}}
child[grow=120]{node[hollow node]{} edge from parent node[xshift=-10]{$2$}}
edge from parent node[xshift=-10]{$2$}
};
\end{tikzpicture}
    \caption{Labeled tree representation of the integer $320=2^{6}\cdot 5=2^{2\cdot 3}\cdot 5$}
    \label{fig:my_label}
\end{figure}
In a successive step, one may forget about labels and consider only rooted trees. Some motivations behind this choice have been discussed in \cite{CC} and include a generalization of the classical density problems for primes and primes configurations like, for instance, twin primes.
It is also interesting to distiguish trees according to their planar or nonplanar structure. We denote by $t(n)$ the planar rooted tree associated to $n$, and by $t^\#(n)$ its nonplanar version.
For instance, $t(12) \neq t(18)$, but $t^\#(12) = t^\#(18)$. In fact $12=2^2 \cdot 3$ and $18=2 \cdot 3^2$ provide two different planar trees with the same nonplanar structure.
The simplest nontrivial tree, namely a single edge attached to the root, corresponds to a prime number. The study of finite patterns of such rooted trees ({\it observables}, for short) therefore generalizes the problem of the distribution of (patterns of) prime numbers, a central topic in analytical number theory. 
In \cite{CC} it has been illustrated that even the simplest questions raised from this perspective are very challenging and often related to very famous % old theorems 
results and conjectures in number theory.

Given the natural numbers $n_1 < n_2 < \cdots < n_k$,
the notation $O(n_1, n_2, \ldots, n_k)$ will represent the pattern formed by the tree $t(n_1)$, followed by $n_2 - n_1 + 1$ arbitrary trees, followed by $t(n_2)$, followed by $n_3 - n_2 + 1$ arbitrary trees and so on and so forth until reaching $t(n_k)$. Thus, in this pattern, $k$ trees are given (including the first from the left and the very last one on the right), while allowing any possible trees in the remaining intermediate slots. 
%$\{n_1 + 1, n_1 + 2, \ldots, n_k - 1\}\setminus \{n_2,\ldots, n_{k-1}\}$.
% $n_1 + 1, \ldots, n_2 - 1, n_2 + 1, \ldots, n_k - 1$.
Some observables are unique, like $O(2,3)$ for instance. Of course, if $O(n_1,\ldots,n_k)$ is unique then the same is true for any $O(m_1,\ldots,m_h)$ when $\{n_1,\ldots,n_k\}$ is a subset of $\{m_1,\ldots,m_h\}$. This motivates the definition of {\it milestone} as any such pattern that is unique, i.e. it occurs only once, and it is minimal, meaning that it does not contain any ``smaller'' similar pattern that is also unique \cite{CC}. 
% It is clear that 
On the other hand, if an observable is not unique it becomes interesting to determine if it occurs infinitely many times or not, and in this last case, how many times it occurs. Perhaps the most famous open question of this sort is about $O(3,5)$, i.e. the infiniteness of the number of twin primes (see \cite{CC} for many other examples). 

% In the same paper, the notion of {\it milestone} had been introduced, loosely speaking an observable
% (finite) pattern of trees
% that occurs only once and it is minimal with this property. 
% XXXX
% For instance, 
It is immediately seen that every observable $O(2,p)$, with $p$ a 
% an odd
prime number,
% i.e.
%$$\vert \underbrace{* * * \ldots *}_{p-3}\vert $$
%where each $*$ denotes an arbitrary tree,
% (this is nothing but $O(2,p)$ with the above notation) 
is a milestone. We say that any such milestone is {\it trivial}. 
Of course, the distribution of trivial milestones is described by the distribution of prime numbers. In \cite{CC} it was proven, among other things, the existence of infinitely many nontrivial milestones.
In order to do that, the authors associated to every integer $n$ the least natural number $\kappa_+(n)$ with the property that the observable $O(n,n+1,\ldots,n+\kappa_+(n))$ occurs only once. 
Indeed, such a $\kappa_+(n)$ exists, for all $n$, and thus the function $\kappa_+: {\mathbb N} \to {\mathbb N}$ is well-defined.\\
Moreover, due to the proof of \cite[Theorem 2.1]{CC}, for all functions $f$ that distinguish primes from any other number (for example one can take $f= \Omega$, the number of prime divisors with multiplicities) one can show that, for a given $n$, there is a number $k$ such that if we have $$(f(n), f(n+1), \ldots, f(n+k)) = (f(m), f(m+1), \ldots, f(m+k)),$$
for some $m$,
then $n=m$. Define the minimal such $k$ by $\kappa_f(n).$ It follows from the proof of \cite[Theorem 2.1]{CC} and the Siegel-Walfisz Theorem that $$\log \kappa_f(n) \leqslant c(\varepsilon)n^\varepsilon,$$
where $c(\varepsilon)>0$ is a constant depending only on $\varepsilon>0$.
In this paper we focus on the function $\kappa_+(n)$ and using elementary arguments we show that actually
$$\kappa_+(n) = O(n\log n),$$
which is better than the bound above.

\medskip
Indeed, the central topic of this paper is to provide a number of results that describe in various ways the occurrence of certain configurations of trees. Our main results include asymptotic bounds for the function $\kappa_+(n)$, both from above and from below. In particular, we show that the function $\kappa_+$ is unbounded, as predicted in \cite{CC}.
Along the way, we also exhibit a 
% constructive identification of a 
quite natural new family of nontrivial milestones (Section 2).
% 2   - 
As it happens for the prime numbers, if an observable occurs infinitely many times then the distance between two consecutive occurrences is unbounded.
% limsup occorrenze consecutive di un albero è illimitato.
Moreover, for any given integer $k \geq 1$ there exists at least one observable of the form $O(n,n+1,\ldots,n+k)$ for some $n$ which occurs infinitely many times (Section 3).
% 3    - 
    % in ogni successione aritmetica lo stesso albero si ripete infinite volte 
Later we find that in every aritmetic progression $a + k q, k \in {\mathbb N}$ with $a$ and $q$ coprime any given planar tree appears infinitely many times, a version of Dirichlet theorem for trees (Section 4). 
%
%4
%    - stima su stringhe di numeri consecutivi con lo stesso albero (simile allo studio di Erdos di stringhe con lo stesso omega). cosa sappiamo dire sulla lunghezza di queste stringhe.
%
%Looking at sequences of consecutive natural numbers with the same nonplanar tree, 
We also deduce from an analogous result by Erd\H{o}s et al. \cite{EPS} stated for $\omega(n)$ an asymptotic estimate for sequences of consecutive integers
with equal nonplanar trees (Section 5).
%we provide an asymptotic estimate analogous to the one obtained by Erd\"os et al. \cite[Theorem 5]{EPS} about sequences of consecutive natural numbers with the same value of $\omega$. 
Here, $\omega(n)$ denotes the number of prime divisors of $n$ (without multiplicity). 
%5      - 
 %   la struttura di un osservabile diverso da una pietra miliare implica vincoli sui parametri rilevanti.
Finally, we discuss how the repeated occurrence of an observable that is not a milestone implies 
% constraints 
specific bounds between the relevant parameters (Section 6).
% inputs

\medskip
% {\bf methods, pls Vitalii can you rewrite what follows here?}\\
Concerning the methods, we use the Richert-Halberstam inequality instead of the method of contour integration when proving the existence of "typical" values of the additive function $E(n)$ (the number of edges of the corresponding tree $t(n)$). From here, we deduce that $\kappa_+(n)$ is unbounded.

To prove that there are configurations of trees that occur infinitely often we use the "lower" version of the Brun's sieve. %In the discussion of the appearence of the same trees on consecutive numbers we use the ingenious technique from Spataru's paper.
Finally, when we study relations between parameters of same configurations of trees (the length of the configuration and the first elements) we use some information about remainder term in the problem
$$\sum_{n\leqslant x}{2^{\omega(n)}} = c_1 x\log x + c_2 x + R(x),$$
using technique from the paper \cite{Spa23}, and one lemma from the theory of combinatorics on words (which is quite unexpected).
%It is quite possible that this estimate could be improved, but we leave this to future work.
To keep the paper more elementary we use a Mertens' style bound $R(x) = \sqrt{x}$ (see \cite{Sa} and \cite{GV}). 
Actually, it follows from the proof of Theorem \ref{MainTh} that if we supposed that
$R(x)\ll x^\delta,$
then for $k$ such that
$$O(m, m+1, \ldots, m+k) = O(n, n+1, \ldots, n+k),\ \ m<n,$$
one can obtain $k\ll n^{{(1+\delta)}/{2}}(\log n)^{O(1)}.$
In particular, it follows from the Riemann's hypothesis  (see Th.1 from \cite{Ba} and take $\delta = 3/8$) that 
$k\ll n^{{11}/{16}}(\log n)^{O(1)}.$
We believe that actually
$k\ll_\varepsilon n^\varepsilon$ for arbitrary small and fixed $\varepsilon >0.$

%Concerning the methods, we would like to mention a couple of very useful Lemmas, one about the study of $\kappa_+$,
%and one about a perhaps unexpected result about configurations of words.

%\medskip
%The techniques employed in our analysis rely for the most part on a number of known results about the function $\omega$. It should be noted, however, that $\omega$ can only account for a relatively limited information about the involved trees. In perspective, it should be important to develope tools that are more sensitive to the full tree structure.

% \medskip
% {\bf summary}    
% The paper is structured as follows...

%\section{Notation}
\bigskip 

Here we collect the main notations used throughout the paper: we denote by $\# A$ the cardinality of the (finite) set $A$. Let the variable $x$ take values in ${\mathbb N}$. We write
$f(x) = O(g(x))$ (or $f(x)\ll g(x)$) for $g(x) > 0$ if $f(x) = O(g(x))$, i.e. exists $C>0$ such that $|f(x)|\leqslant C g(x)$. 
We also write $f(x) = o(g(x))$ if $f(x)/g(x)\to 0$ when $x \to +\infty$.
Let $\textbf{E}(t)$ be the number of edges of a given tree $t$, and let $E(n) = \textbf{E}(t(n))$. 
We denote by $\mathcal{T}$ the set of all finite rooted planar trees, and by $[\mathcal{T}]$ the set of all finite rooted nonplanar trees.
We also need the following functions $\Omega$ and $\omega$, such that 
$$\Omega(n) = \sum_{p|n}\nu_p(n)$$
(the number of prime divisors of $n$ honoring their multiplicity),
where $\nu_p(n)$ is the exponent of the prime $p$ in the canonical decomposition of $n$, and
$$\omega(n) = \sum_{p | n} 1 $$
(the number of prime divisors of $n$ without multiplicity).

\section{Some bounds for $\kappa_+(n)$}

In this section we prove that the function $\kappa_+(n)$ is unbounded. Moreover, we prove the following
\begin{theorem}\label{Th1.1} Let $x$ be large enough  and $1\leqslant k \leqslant \log x/(5\log\log x)$. Then there are $m\geqslant x^{1/10}$ numbers
	$n_1, n_2, \ldots, n_m\leqslant x$ such that $\kappa_+(n_1), \kappa_+(n_2), \ldots, \kappa_+(n_m)\geqslant k.$ 
\end{theorem} 
%Hence,  we immediately get that
%$$K(x) = \max_{n\leqslant x}\kappa_+%(n)\geqslant \dfrac{\log x}{5 \log\log x}.$$
%Which better than the previous bound %$K(x)\geqslant \sqrt{\log x}.$
 We need the following lemma.
\begin{lemma}\label{L1.2} For $x$ large enough
	we have
	$$\#\{n \leqslant x: E(n)\geqslant 3 \log\log x\} \ll \dfrac{x}{(\log x)^\mu},$$
	where $\mu = 3\log 2-1.$
\end{lemma}
\begin{proof}
	We have
	$$\sum\limits_{\substack{n \leqslant x \\ E(n)\geqslant 3\log\log x}} 1 \leqslant 2^{-3\log\log x}\sum_{n \leqslant x}2^{E(n)}.$$
	It remains to prove that the last sum is at most $O(x\log x).$
	Note that $2^{E(\cdot)}$ is a multiplicative function. Then Richert - Halberstam inequality (see \cite{HT}, Theorem 01) implies that
	$$\sum_{n \leqslant x}2^{E(n)} \ll \dfrac{x}{\log x}\sum_{n\leqslant x}\dfrac{2^{E(n)}}{n}\leqslant \dfrac{x}{\log x}\prod_{p\leqslant x}\left( 1+\dfrac{2^{E(p)}}{p}+\dfrac{2^{E(p^2)}}{p^2}+\cdots\right), $$
	where $p$ runs through primes up to $x$. Further, we have
	\begin{multline*} 1+\dfrac{2^{E(p)}}{p}+\dfrac{2^{E(p^2)}}{p^2}+\cdots \\
		= \left(1-\dfrac{1}{p} \right)^{-2}\left\{\left( 1+\dfrac{2^{E(p)}}{p}+\dfrac{2^{E(p^2)}}{p^2}+\cdots\right)\left(1-\dfrac{1}{p} \right)^{2} \right\} = \left(1-\dfrac{1}{p} \right)^{-2}P(p), \ \ \text{say.}
	\end{multline*}
	Now we estimate the term $P(p)$. Since $E(p^{\nu}) = E(\nu)+1,$ it follows that
	\begin{equation*}
		P(p) = \left( 1+ \dfrac{2}{p}+\sum_{\nu\geqslant 2}\dfrac{2^{E(\nu)+1}}{p^\nu}\right)\left(1-\dfrac{2}{p}+\dfrac{1}{p^2} \right) = 1+\dfrac{1}{p^2}+O\left(\sum_{\nu\geqslant 3}\dfrac{2^{E(\nu)}}{p^\nu} \right).  
	\end{equation*}
	Since $E(\nu)\leqslant \Omega(\nu) \leqslant \log \nu/\log 2,$ we get
	$$P(p) = 1 + \dfrac{1}{p^2}+ O\left( \sum_{\nu\geqslant 3}\dfrac{\nu}{p^\nu}\right) = 1 + O\left(\dfrac{1}{p^2}\right). $$
	Hence, using Mertens' third theorem we obtain
	$$\prod_{p\leqslant x}\left( 1+\dfrac{2^{E(p)}}{p}+\dfrac{2^{E(p^2)}}{p^2}+\cdots\right)\leqslant \prod_{p\leqslant x}\left( 1- \dfrac{1}{p}\right)^{-2}\prod_{p} P(p)\ll (\log x)^2.$$
	Therefore, we get
	$$\sum_{n\leqslant x}2^{E(n)}\ll x\log x,$$
	which gives the lemma.
\end{proof}
For $x\geqslant 2$ and $k\geqslant 1$ we define $J_k(x)$ to be
	$$J_k(x) = \#\left\{n\leqslant x: E(n), E(n+1), \ldots, E(n+k) \leqslant 3 \log\log x\right\}.$$
  We prove the following
\begin{lemma}\label{1.2}
Uniformly on $1\leqslant k\leqslant \log x$ we have
 $$J_k(x)= x\left(1+o(1)\right),\ \ (x \to +\infty).$$
\end{lemma}
\begin{proof}
Indeed, on the one hand
$J_k(x)\leqslant x.$
	On the other hand, using Lemma \ref{L1.2} we obtain
	\begin{multline*}
		J_k(x) = \sum_{n\leqslant x} 1 - \sum\limits_{\substack{n\leqslant x \\ \exists j\ (0\leqslant j\leqslant k): \\ E(n+j) > 3\log\log x}} 1 \geqslant \lfloor x\rfloor -
 \sum_{j=0}^k\sum\limits_{\substack{n\leqslant x \\ E(n+j)>3\log\log x}}1 \\
 \geqslant x - O\left(1 + k\sum\limits_{\substack{n\leqslant x+k\\ E(n)>3\log\log x}}1 \right) = x - O\left(1 + k\sum\limits_{\substack{n\leqslant x\\ E(n)>3\log\log x}}1 +k^2 \right) \\
		= x - O\left(\dfrac{k x}{(\log x)^{3\log2-1}}\right) = x\left(1-O\left( \dfrac{1}{(\log x)^{3\log 2 - 2}} \right) \right). 
	\end{multline*}
	This concludes the proof.
 \end{proof}
	Now we are ready to prove Theorem \ref{Th1.1}.
	
 \begin{proof}
 Set 
	$$T(x) = \#\left\{t(n) \in \mathcal{T}: n\leqslant x, E(n)\leqslant 3\log\log x\right\}.$$
	Since each tree $t\in \mathcal{T}$ can be encoded as a sequence of $0$ and $1$ of size $2\textbf{E}(t)$, we conclude that
	$$\#\left\{t\in \mathcal{T}: \textbf{E}(t) = \nu \right\}  \leqslant 4^\nu.$$
	Hence, we find that
	$$T(x)\leqslant \sum_{\nu\leqslant 3\log\log x} 4^\nu \ll (\log x)^{3\log 4}.$$
	Put 
	$$T = \left\{(t(n), \ldots, t(n+k)) \in \mathcal{T}^{k+1}: n\leqslant x; E(n+j)\leqslant 3\log\log x, 1\leqslant j\leqslant k\right\}.$$
	It is easy to see that
	$$\# T\leqslant T(x+k)^{k+1}\ll (\log x)^{3\log 4\,(k+1)(1+o(1))}.$$
	Lemma \ref{1.2} and the pigeonhole principle imply that there is at least one tuple $\overline{t} \in T$ that occurs $m\geqslant \lceil J_k(x)/\#T\rceil$ times. Since $k\leqslant \log x/(5 \log\log x)$, it follows that there are
	\begin{multline*}
 m\geqslant \dfrac{J_k(x)}{\# T}\geqslant \dfrac{x(1+o(1))}{(\log x)^{3\log 4 (k+1)(1+o(1))}}\\ = \exp{\left((1+o(1))(\log x-3\log 4\,(\log\log x) k) \right)}\\
 \geqslant \exp\left(\log x\left(1-\frac{3\log 4}{5}\right)\right)
 \geqslant \exp(\log x/10)
 \end{multline*}
	numbers $n_1, n_2, \ldots, n_m \leqslant x,$ for which 
	$$\overline{t} = (t(n_1), t(n_1+1), \ldots, t(n_1+k)) = \cdots = (t(n_m), t(n_m+1), \ldots, t(n_m+k)).$$
	Hence, we obtain $\kappa_+(n_1), \kappa_{+}(n_2), \ldots, \kappa_+(n_m)\geqslant k.$ The claim follows.

\end{proof}
Note that Theorem \ref{Th1.1} implies that there are infinitely many integers $n$ such that
$$\kappa_+(n)\gg \dfrac{\log n}{\log\log n}.$$
In particular, the function $\kappa_+(n)$ is unbounded.
 Now we give some upper bound for $\kappa_+(n)$. Namely, we prove the following
 \begin{theorem}\label{Th1.4}
     For $n\to +\infty$ we have
     $$\kappa_+(n)\leqslant 2n\log n(1+o(1)).$$
 \end{theorem}
%\begin{equation}\label{K_up_bound}
	%K(x)\leqslant 2x\log x(1+o(1)).
	%\end{equation}
We need the following two lemmas.
\begin{lemma}
	For all $n, m \geqslant 1$ we have
	$$\kappa_+(n)\leqslant \kappa_+(n+m)+m.$$
	\end{lemma}
\begin{proof}
This immediately follows from the definition of the function $\kappa_+(n)$.
\end{proof}
%\begin{lem}
%	Let $n\geqslant 2$ be square-free and
%	$\omega(n)>\!\!\max\limits_{n+1\leqslant k\leqslant 2n-1}\omega(k),$
%	then $$|O(n, n+1, \ldots, 2n-1)|=1.$$
%\end{lem}
%\begin{proof}
	%Let $m$ be integer such that
%	$$O(n, n+1, \ldots, 2n-1) = O(m+n, m+n+1, \ldots, m+2n-1).$$
%	Then one and only one number among the integers $\left\{m+n, m+n+1, \ldots, m+2n-1\right\}$ is divisible by $n$.
%If $j\geqslant 1,$ then $\omega(m+n+j) = \omega(n+j)<\omega(n)$,
%and thus $$m+n+j\not{\equiv} 0 \pmod{n}.$$ Hence, $m+n \equiv 0 \pmod{n}.$
%Since $m+n$ is square-free, and $\omega(n+m) = \omega(n)$, it follows that $m=0$. The claim follows.
\begin{lemma}
	Let $k\geqslant 2$ be an integer, and let $q_k = p_1 p_2\cdots p_k$ be the product of the first $k$ primes. Then
	$$|O(q_k, q_k+1, \ldots, 2q_k-1)| = 1.$$
\end{lemma}
\begin{proof}
	Since $2q_k-1<q_{k+1}$, it follows that $\omega(q_k+j)\leqslant k$ for $0\leqslant j\leqslant q_k-1.$ Note that if $\omega(q_k+j) = k$, then $q_k+j$ is square-free. Indeed, assume the converse. Then $$q_k+j\geqslant p_1^2p_2\cdots p_k>2q_k-1;$$
	this is a contadiction.
	It remains to prove that if
	$$O(q_k, q_k+1, \ldots, 2q_k-1) = O(n, n+1, \ldots, n+q_k-1),$$
	then $n=q_k.$
	Since $r = q_k$ is the smallest number with the condition $\omega(r)=k,$ we have $n\geqslant q_k$. Further, there is exactly one number among the integers $\left\{n, n+1, \ldots, n+q_k-1\right\}$ such that it is divisible by $q_k$. Suppose that 
	$$n+j\equiv 0\pmod{q_k}$$
	for some $j\geqslant 0.$ Then $\omega(n+j)=\omega(q_k+j) = k.$ Hence, $n+j$ is square-free. Then $n+j = q_k$, which is possible if and only if $j=0$ and $n=q_k.$ The claim follows.
\end{proof}
Now we prove Theorem \ref{Th1.4}.

\begin{proof}
Choose $k$ such that $q_k\leqslant n <q_{k+1}.$ The Prime Number Theorem implies that $q_k = \exp\left(p_k(1+o(1))\right) $. Hence, $p_k\leqslant \log n(1+o(1))$, and $$p_{k+1} = p_k(1+o(1))\leqslant \log n(1+o(1)).$$
Using the above two lemmas we obtain 
$$\kappa_{+}(n)\leqslant \kappa_+(q_{k+1})+q_{k+1}-n\leqslant  2q_{k+1}=2q_kp_{k+1}\leqslant 2n\log n(1+o(1)).$$

\end{proof}
\section{The distance between same trees}

We prove that the distance between the same observed configurations of trees can be arbitrarily large. We need the following lemma.
\begin{lemma}\label{lem 1.4}
	Let $t \in [\mathcal{T}]$ be some tree, and let $\{n_l\}_{l=1}^{+\infty}$ be the sequence of consecutive numbers for which $t(n_l) = t$. Suppose that $n_1\leqslant \frac{9}{10}n_k$ for some $k\geqslant 2$, then the following inequality holds
	$$\max_{l\leqslant k-1}\left( n_{l+1}-n_{l}\right) \geqslant \dfrac{c_0\log n_k (\omega(n_1)-1)!}{(\log\log n_k+c_1)^{\omega(n_1)-1}},$$
	where $c_0, c_1>0$ are some absolute constants.
	\end{lemma}
\begin{proof}
	%We have
%	$$\sum_{l\leqslant k-1}(n_{l+1}-n_l) = n_k - n_1\geqslant 0.1 n_k.$$
Using the Hardy-Ramanujan inequality (see \cite{HR17}), we obtain
$$k = \sum\limits_{\substack{n\leqslant n_k \\ t(n) = t}}1\leqslant \sum\limits_{\substack{n\leqslant n_k \\ \omega(n) = \omega(n_1)}}1 \leqslant \dfrac{c_2 n_k (\log\log n_k+c_1)^{\omega(n_1)-1}}{\log n_k (\omega(n_1)-1)!},$$
where $c_1, c_2>0$ are some absolute constants.
Hence,
\begin{multline*}
	0.1n_k\leqslant n_k - n_1 = \sum_{l\leqslant k-1}(n_{l+1}-n_l)\leqslant k \max_{l\leqslant k-1}\left( n_{l+1}-n_{l}\right)\\
	\leqslant   \dfrac{c_2 n_k (\log\log n_k+c_1)^{\omega(n_1)-1}}{\log n_k (\omega(n_1)-1)!}\max_{l\leqslant k-1}\left( n_{l+1}-n_{l}\right).
	\end{multline*}
Then, we have
$$\max_{l\leqslant k-1}\left( n_{l+1}-n_{l}\right) \geqslant \dfrac{\log n_k (\omega(n_1)-1)!}{10c_2(\log\log n_k+c_1)^{\omega(n_1)-1}}.$$
The proof of the lemma is complete.
\end{proof}

\begin{theorem}
	Suppose that 
	$$|O(n_1, n_1+h_1, \ldots, n_1+h_k)| = +\infty,$$
	where $n_1\geqslant 1$, and $1\leqslant h_1< h_2<\cdots <h_k.$ Let $\{n_l\}_{l=1}^{+\infty}$ be the sequence of consecutive numbers for which
	\begin{multline}O(n_1, n_1+h_1,\ldots, n_1+h_k) = O(n_2,n_2+h_1, \ldots,n_2+h_k) = \cdots\\
		 = O(n_l, n_l+h_1,  \ldots,  n_l+h_k) = \cdots,
		\end{multline}
	then we have
	\begin{equation}\label{limsup}
		\limsup\limits_{l \to +\infty}\left(n_{l+1}-n_l -h_k\right) = +\infty.
		\end{equation}
\end{theorem}
\begin{proof}
	Indeed, the sequence $\{n_l\}_{l=1}^{+\infty}$ is a subsequence of $\{N_l\}_{l=1}^{+\infty}$, where $N_1, N_2, \ldots, N_l, \ldots$ are consecutive numbers for which $t(N_l) = t(n_1),\ (l\geqslant 1).$ Lemma \ref{lem 1.4} implies that
	$$\limsup\limits_{l \to +\infty} (N_{l+1}-N_l) = +\infty.$$
	Then, since $h_k$ is fixed, we obtain \eqref{limsup}. The claim follows.
\end{proof}
\begin{theorem}
	For all $k\geqslant 1$, there is a natural number $n\geqslant 1$, such that
	$$|O(n, n+1, \ldots, n+k)| = +\infty.$$
\end{theorem}
\begin{proof}
	Let us fix the value $k\geqslant 1$ and let $a=a(k)$ be chosen later. Consider the sum
	$$V_k(x) =\sum\limits_{\substack{n\leqslant x \\ \Omega(n+j)\leqslant a,\ 0\leqslant j\leqslant k}} 1.$$
	It is enough to show that this sum is unbounded. Indeed, since \begin{multline*}\#\left\{t\in\mathcal{T}: \exists n\geqslant 1, \Omega(n)\leqslant a, t(n)=t\right\}\\
		\leqslant \sum_{m\leqslant a}\#\left\{(x_1, \ldots, x_m)\in\mathbb{N}^m: x_1+\cdots+x_m\leqslant a\right\} = O_a(1),
		\end{multline*}
	it follows that there are only a finite number of tuples $(t_1, \ldots, t_{k+1})\in \mathcal{T}^{k+1}$ such that
	$t_j = t(n+j), \Omega(n+j)\leqslant a$ for some $n\geqslant 1$ and for each $0\leqslant j\leqslant k$. Hence, if $V_k(x)$ is unbounded as $x\to +\infty$, then at least one of the tuple of trees occurs infinitely many times. Define the number $q=q(k)$ as follows:
	$$q=\prod_{p\leqslant k+1}p^{\alpha(p)},$$
	where $\alpha(p)$ is the smallest number $\alpha$ so that
	$p^{\alpha}>k+1.$
	Now we are going to estimate $V_k(x)$. Put $$P(x) = \prod_{p\leqslant x}p,$$ and leave only those $n$ that are multiples of $q$. Then we get
	\begin{equation}
		V_k(x)\geqslant \sum\limits_{\substack{n\leqslant x-k\\ n \equiv 0\pmod q \\ \Omega(n+j)\leqslant a,\, 0\leqslant j\leqslant k}}1 \geqslant \sum\limits_{\substack{l\leqslant y\\ \gcd\left( {l, P(k+1)}\right)=1 \\ \Omega(ql+j)\leqslant a,\, 0\leqslant j\leqslant k}}1,
	\end{equation}
where $y=(x-k)/q$. It is easy to see that every number $1\leqslant j\leqslant k$ divides $q$. Put $a_j = q/j$, then we will have $lq+j = j(a_jl+1)$. Choose $a$ in the form
$$a = 10k + \sum\limits_{p\leqslant k+1}\alpha(p).$$
Then since $$\Omega(j)\leqslant \Omega(q) = \sum\limits_{p\leqslant k+1}\alpha(p),$$
it follows that
$$V_k(x)\geqslant \sum\limits_{\substack{l\leqslant y\\ \gcd\left( {l, P(k+1)}\right)=1 \\ \Omega(l)\leqslant 10k \\ \Omega(a_jl+1)\leqslant 10k,\, 1\leqslant j\leqslant k}}1\geqslant\sum\limits_{\substack{l\leqslant y\\ \gcd\left(F(l), P(x^\varepsilon)\right) = 1 }}1= S_k(x),$$
where $\varepsilon = \varepsilon(k) = {1/(10k)}$ and
$$F(l) = l \prod\limits_{1\leqslant j\leqslant k}(a_j l+1).$$
The last inequality is explained as follows. Let $\gcd\left(m, P(x^\varepsilon)\right)=1$ for some $m\leqslant x$ and let
$$m = p_1 p_2\cdots p_s$$ 
% is 
be a decomposition of $m$ into not necessarily different prime factors. %where $p_1\geqslant p_2\geqslant \cdots\geqslant p_s$, 
Then $s = \Omega(m)$ and 
$$x\geqslant m > x^{\varepsilon s} = x^{\frac{s}{10k}}.$$
Hence, we get that $\Omega(m) = s<10k$, and the last inequality follows. Now, using the Brun's sieve (see Th. $6.2.5$ from \cite{Mur05}), we will give the lower bound for $S_k(x)$. For a finite sequence of natural numbers $\mathcal{A} = (a)$, a subset of primes $\mathcal{P}$, and a parameter $z$, we define the sum $S(\mathcal{A}, \mathcal{P}, z)$ to be
$$S(\mathcal{A}, \mathcal{P}, z) = \sum\limits_{\substack{a \in \mathcal{A} \\ \gcd\left(a, P\right) = 1}} 1,$$
where $$P = \prod\limits_{p \in \mathcal{P},\,p\leqslant z}p.$$
Let also $\mathcal{A}_d = \{a \in \mathcal{A}: a \equiv 0 \pmod d\}$. We assume that, for all $d|P$, 
\begin{equation}\label{cond_g}
	\#\mathcal{A}_d = X\dfrac{w(d)}{d}+R_d,
\end{equation}
where $X$ is indended of $d$, and $w(d)$ is a multiplicative function such that $0\leqslant w(p)<p$ for $p \in \mathcal{P}$ and $w(p)=0$ for $p \not \in \mathcal{P}$.

Put $\mathcal{A} = \left\{F(l)\right\}_{l\leqslant y}$. Let $\mathcal{P}$ be the set of all prime numbers, and let $z = x^\varepsilon$. Then we conclude that $S_k(x) = S(\mathcal{A}, \mathcal{P},z)$. Let us check condition \ref{cond_g}. Consider a number $d|P = P(x^\varepsilon)$. Fix an arbitrary prime divisor $p$ of $d$ and consider the sum
$$S_p = \sum\limits_{\substack{l\leqslant p \\ F(l)\equiv 0\pmod p}} 1.$$
If $p\leqslant k+1$, then the congruence $F(l)\equiv 0\pmod p$ is equivalent to $l \equiv 0\pmod p$ because $p$ divides $a_j$ for each $1\leqslant j\leqslant k$. Hence, we obtain $S_p = 1$. Note that if $p>k+1$, then each $a_j$ is prime to $p$ and $S_p\leqslant k+1$. Using the Chinese remainder theorem, we obtain
$$\sum\limits_{\substack{l\leqslant d \\ F(l)\equiv 0\pmod d}} 1= \prod\limits_{\substack{p|d \\ p>k+1}}S_p.$$
Hence, splitting the segment $[1, y]$ into segments of length $d$, we find that
\begin{multline}
	\#\mathcal{A}_d = \sum\limits_{\substack{l\leqslant y\\ F(l)\equiv 0\pmod d}} 1 = \sum\limits_{\substack{l\leqslant d\lfloor\frac{y}{d} \rfloor\\ F(l)\equiv 0\pmod d}} 1 +  \sum\limits_{\substack{ d\lfloor\frac{y}{d} \rfloor<l\leqslant y\\ F(l)\equiv 0\pmod d}} 1\\
	= \left\lfloor\frac{y}{d} \right\rfloor\prod\limits_{\substack{p|d \\ p>k+1}}S_p + \theta \prod\limits_{\substack{p|d \\ p>k+1}}S_p = \left( \frac{y}{d}-\theta_1\right)  \prod\limits_{\substack{p|d \\ p>k+1}}S_p + \theta \prod\limits_{\substack{p|d \\ p>k+1}}S_p\\
	=\frac{y}{d} \prod\limits_{\substack{p|d \\ p>k+1}}S_p + \left( \theta-\theta_1\right)\prod\limits_{\substack{p|d \\ p>k+1}}S_p,
	\end{multline}
where $0\leqslant \theta, \theta_1\leqslant 1$. Thus, we conclude that $X = y$, $w(d) = \prod\limits_{\substack{p|d \\ p>k+1}}S_p$, and $$R_d = \left( \theta-\theta_1\right)\prod\limits_{\substack{p|d \\ p>k+1}}S_p.$$
Note that $|R_d|\leqslant w(d)$. For $p\leqslant k+1$ we have 
$$0\leqslant \dfrac{w(p)}{p}= \dfrac{1}{p}\leqslant \dfrac{1}{2},$$
and for $p>k+1$:
$$0\leqslant \dfrac{w(p)}{p}\leqslant \dfrac{k+1}{p}\leqslant \dfrac{k+1}{k+2} =  1 - \dfrac{1}{k+2}.$$
Finally, from the Mertens' first theorem, we see that for $2\leqslant u<z$ we have
$$\sum_{u\leqslant p<z}\dfrac{w(p)\log p}{p}\leqslant(k+1)\log\dfrac{z}{u}+3(k+1).$$
Therefore, choosing parameters $A_1 = k+2$, $\kappa = k+1,$ $A_2 =  3(k+1)$, and applying Theorem $6.2.5$ from \cite{Mur05}, we conclude that
$$S_k(x)\geqslant y W(z)(1-E)-R,$$
where 
$$W(z) = \prod_{p\leqslant z}\left(1-\dfrac{w(p)}{p}\right) = \prod_{p\leqslant k+1}\left(1-\dfrac{1}{p}\right) \prod_{k+1<p\leqslant x^\varepsilon}\left(1-\dfrac{S_p}{p}\right),$$
$$E = \dfrac{2\lambda^{2b}e^{2\lambda}}{1-\lambda^22^{2+2\lambda}}\exp\left(\left(2b+3\right)\dfrac{c_1}{\lambda \log z}\right),\ \ c_1 = O_{ \kappa, A_1, A_2}(1),$$
and
$$R = O(z^{2b-1+\frac{2.01}{e^{2\lambda/\kappa}-1}}).$$
Hence, choosing  $\lambda = 0.253$, $b=1$ and using the fact that $x\geqslant x_0(k)$, we get $0<E\leqslant 0.99$,
$$W(z)\gg_k \prod_{k+1<p\leqslant x^\varepsilon}\left(1-\dfrac{k+1}{p}\right)\gg_k \prod_{p\leqslant x^\varepsilon}\left(1-\dfrac{1}{p}\right)^{k+1}\gg_k \dfrac{1}{(\log x)^{k+1}},$$
and
$$R\ll x^{\varepsilon\left(1+\frac{2.01}{\exp({0.506}/(k+1))-1}\right)}\leqslant x^{\varepsilon\left(1+\frac{2.01(k+1)}{0.506}\right)}\leqslant x^{\frac{4k+5}{10k}}\leqslant x^{0.9}.$$
Thus, we obtain that
$$V_k(x)\geqslant S_k(x)\gg_k\dfrac{x}{(\log x)^{k+1}}\to +\infty,$$
as $x\to +\infty.$ The proof of the theorem is complete.
\end{proof}

\section{Trees in arithmetic progressions}
The following theorem is a natural generalization of the well known result by Dirichlet about primes in arithmetic progressions.
\begin{theorem}
	For any fixed nontrivial planar rooted tree $t$, any modulus $q$, and any $a$ coprime to $q$, there are infinitely many $n$ such that $n \equiv a\pmod q$ and $t(n) = t.$
	\end{theorem}
\begin{proof}
	Let $t$ contain $s \geq 1$ edges incident to the root. Let $t_1$ be the first subtree connected to the root from the left, 
 $t_2$ be the second one, and so on. Define $\delta$ to be the multiplicative order of $a$ modulo $q$, that is
	$$\delta = \min\{\nu\geqslant 1: a^\nu \equiv 1 \pmod q\}.$$	
	 We will show that we can label each edge of $t$ by some prime number so that the resulting number is equal to $a$ modulo $q$. We label the first edge connected to the root on the left by a prime number $p_1 \equiv a \pmod q$, and the $j^{\text{th}}$ edge connected to the root (here $2\leqslant j \leqslant s$) by a prime $p_j\equiv 1 \pmod q$ so that $p_1<p_2<\cdots<p_s$. 
Now we label each edge of $t_1$ by a prime $ p \equiv 1 \pmod \delta$ and each edge of $t_j$ for $2\leqslant j \leqslant s$ by an arbitrary prime. It follows from the Dirichlet's theorem on arithmetic progressions that we can also achieve that $p<q$ whenever the edges corresponding to $p$ and $q$ are incidental to a common vertex and the edge corresponding to $p$ is located to the left of the edge corresponding to $q$.
Denote the number we have constructed by $n$, then $t(n) = t$. Moreover, it is easy to see that for some $N\geqslant 1$ we have
$$n\equiv a^{1+\delta N} \equiv a \pmod q.$$
Since we can construct infinitely many such numbers $n$, we conclude the proof.
\end{proof}

\section{Same tree on consecutive numbers}

Consider two numbers $n,n+1$ with the same (possibly planar) tree. For instance, $O(2,3)$ is a milestone. It is less obvious, but true, that $O(8,9)$ is also a milestone: indeed, if $O(n,n+1) = [t,t]$ for some rooted tree $t$ with only one edge connected to the root (i.e., $\omega(n)=1 = \omega(n+1)$) then $n=8$ \cite{CC}. Also, in this case the corresponding tree $t$ has no symmetry so there is no distinction between the planar and the nonplanar situation. On the other hand, $O(14,15)$ is not a milestone. It is more difficult to decide if the number of pairs of consecutive integers that are both squarefree (i.e. discrete) semiprimes is actually infinite, i.e. if $|O(14,15)|=+\infty$.
By Theorem 4 in \cite{GGPPSY}.
% Small gaps between almost primes, the parity problem, and some conjectures of Erdős on consecutive integers II, Goldston, Daniel A.; Graham, Sidney W.; Panidapu, Apoorva; Pintz, Janos; Schettler, Jordan; Yıldırım, Cem Y., J. Number Theory 221 (2021), 222–231.
there are infinitely many pairs $(n,n+1)$ such that $t^\#(n) = t^\#(n+1) = t^\#(420)$. 
%It is unclear whether the same is true for at least one of the corresponding planar trees. What about their distribution? 
Concerning the $k=3$ case, the sequence A039833 in the OEIS accounts for the smallest of three consecutive squarefree numbers with exactly two prime divisors (again, the corresponding tree is symmetric).
The first such example is $n=33$ and one finds 2377 instances of that observable in the window up to one million although we do not know if $|O(33,34,35)| = + \infty$.
% In an appendix, we list all such examples up to 1,000,000.
% add appendix
%What about other examples with a different (possibly planar) tree?
%Do we have any example of $k$ consecutive integers with the same nonplanar tree, for $k = 4, 5, \ldots$?
%For $k=4$, the first such example is given by integers $n + j$, $j=0,\ldots,3$ with $n = 19940$,
%while for $k =7$, the first example is given by $n+j$, $j=0,\ldots,6$ with $n=30530822$. 
% \end{remark}

Define a sequence by setting $a_n$ to be the least integer with the property that $a_n, a_n + 1, a_n + 2, \ldots, a_n + n-1$ share the same nonplanar tree (independently of the tree itself). Then the first entries are $a_2 = 2$, $a_3 = 33$, $a_4 = 19940$, $a_5 = 136824$, $a_6 = 630772$, $a_7 = 30530822$.

\medskip
\noindent In the sequel, we look for bounds on the number of equal consecutive nonplanar trees.

\medskip

For $n\geqslant 1$ define the value $k(n)$ by
$$k(n) = \max\left\{k\geqslant 1: t^\#(n+1) = t^\#(n+2) = \cdots = t^\#(n+k)= t\ \text{for some}\  t \in \left[\mathcal{T}\right]\right\}$$
%We believe that $k(n)$ is bounded uniformly on $n$.
Note that $k(n)$ is bounded for planar rooted trees. Indeed, since $4 \nmid 4n+2$ and $2$ is the smallest prime divisor of both $4n$ and $4n+2$, it follows that  $t(4n+2)\neq t(4n)$. Hence, if
$$t(n+1) = t(n+2) = \cdots = t(n+k) = t$$
for some $t \in \mathcal{T}$, then $k\leqslant 3.$ %Note that for $n=32$ we have $k=3$, as well for $n=84$ and $n=92$.
% Do we have infinitely many such triples, for this specific tree?

In this section we prove the following
\begin{theorem}\label{Th6}
	Let $n$ be large enough and $k\geqslant 2$. Then we have
	$$\log k(n) \leqslant \dfrac{2\log n}{\log\log n}.$$
	Moreover, if $n$ has the form
	\begin{equation}\label{type_n}
 n = (2\uparrow\uparrow m)^{2\uparrow\uparrow (m-1)}, \ \ m\geqslant 2,
 \end{equation}
	where $\uparrow\uparrow$ denotes tetration,
	 we have
	$$\log k(n) \leqslant \sqrt{\log n}.$$
\end{theorem} 
%Define $\log_2^*(n)$ to be the base $2$ iterated logarithm, that is
%$$\log_2^*(n) = \max\left\{m\geqslant 1: 2\uparrow\uparrow m\leqslant n\right\}.$$
% We prove the following theorem.
 %\begin{theorem}
 %	For $n\geqslant 1$ we have
 	%$$k(n) = \exp(O(\log_2^* n)).$$
 %	\end{theorem}
% \begin{proof}
%Note that
 %$$k(n)< 2^{H(n)+1}.$$
 %Indeed, among any consecutive $2^{\alpha+1}$ numbers, there is one number that is divisible by $2^\alpha$ and not divisible by $2^{\alpha+1}$. Hence, if $k(n)\geqslant 2^{H(n)+1}$, then 
% \end{proof}
 First of all, we prove the following lemma.
 
 \begin{lemma}\label{lem6.2}
   For $n \to +\infty$ we have
 	$$k(n) = o(n).$$ 
 \end{lemma}

\begin{proof}
	If $n+1, n+2, \ldots, n+k$ are primes for sufficiently large $n$, then $k=1$, and we are done. Suppose that $n+1, n+2, \ldots, n+k$ are composite. It follows from the Prime Number Theorem that for each $\varepsilon >0$ there is a $n_0(\varepsilon)>0$ such that for all $n\geqslant n_0(\varepsilon)$ the interval $I_\varepsilon = (n, n+\varepsilon n] $ contains a prime number. Thus $k\leqslant \varepsilon n $ as claimed.
\end{proof}
Let $H(n)$ be the height of $t^{\#}(n)$. We need the following lemma.
 \begin{lemma}\label{lem6}
 	For $l\geqslant 1$ and $x, y\geqslant 2$ put
 	$$S_l(x; y) = \#\left\{x<n\leqslant x+y: H(n)\geqslant l\right\}.$$
 	Then we have
 	$$S_l(x; y) \ll \dfrac{y}{2\uparrow\uparrow l} + 2\uparrow\uparrow(l-1)\dfrac{(x+y)^{{1}/{2\uparrow\uparrow(l-1)}}}{\log(x+y)},$$
 	where the implied constant is absolute.
 \end{lemma}
\begin{proof} If $H(n)\geqslant l,$ then there is a prime power $p^r$ such that $H(r)\geqslant l-1$ and $p^r$ divides $n$. Hence, we have
	\begin{equation*}S_l(x;y) \leqslant \underset{H(r)\geqslant l-1} {\underset{p^r\leqslant x+y,\, p\,\text{prime}}{\sum\sum}}\sum\limits_{\substack{x<n\leqslant x+y \\ n \equiv 0\,(\text{mod}\, p^r)}} 1 = S_1 +S_2,
		\end{equation*}	
	where $S_1$ denotes the contribution of those $p$ and $r$ for which $p^r\leqslant y$, and $S_2$ denotes the contribution of the remaining $p$ and $r$.
	
	First we estimate $S_1$. Since $r = 2\uparrow\uparrow (l-1)$ is the smallest number for which $H(r)\geqslant l-1$, we have
	\begin{multline*}
		S_1\leqslant \underset{H(r)\geqslant l-1} {\underset{p^r\leqslant y,\, p\,\text{prime}}{\sum\sum}}\left( \dfrac{y}{p^r}+ 1\right) \leqslant 2y\sum_{p}\sum_{r\geqslant 2 \uparrow\uparrow (l-1)}\dfrac{1}{p^r}
		\leqslant 4y\sum_{p}\dfrac{1}{p^{2\uparrow\uparrow (l-1)}}\leqslant 4y (\zeta(2 \uparrow\uparrow (l-1))-1),
	\end{multline*}
where $\zeta(s)$ denotes the Riemann zeta function. Since
$$\zeta(\alpha)-1 \leqslant \dfrac{1}{2^\alpha}+\int_2^{+\infty}\dfrac{dt}{t^\alpha} \ll \dfrac{1}{2^\alpha},\ \ (\alpha >1), $$
we conclude that $S_1 \ll y/(2\uparrow\uparrow l).$
	
	Now we are going to estimate the sum $S_2$. Note that there is at most one $n$ for which $x<n\leqslant x+y$ and $n \equiv 0 \pmod {p^r},\ p^r>y.$ Hence, we get that
	\begin{equation*}
		S_2\leqslant \sum_{p \leqslant (x+y)^{{1}/{2 \uparrow\uparrow (l-1)}}}\sum\limits_{\substack{r\leqslant \frac{\log(x+y)}{\log p} \\ H(r)\geqslant l-1}}1.
	\end{equation*}
Using the uniformly bound (see \cite{Iud21})
\begin{equation}\label{mainineq}
\sum\limits_{\substack{r \leqslant x \\  H(r)\geqslant k}} 1 \ll \dfrac{x}{2 \uparrow\uparrow k},\ \ (x\geqslant 2, k\geqslant 1),
\end{equation}
we find that
\begin{equation}\label{S2}
	S_2 \ll \dfrac{\log(x+y)}{2 \uparrow\uparrow (l-1)} \sum_{p \leqslant (x+y)^{{1}/{2 \uparrow\uparrow (l-1)}}}\dfrac{1}{\log p}.
	\end{equation}
	Let $\pi(z)$ denotes the number of primes up to $z$, then using Chebyshev's bound $\pi(z) \ll z/\log z$, we obtain
	$$\sum_{p\leqslant z} \dfrac{1}{\log p} \leqslant \sqrt{z}+\sum_{\sqrt z< p\leqslant z}\dfrac{1}{\log p}\ll \sqrt{z} + \dfrac{\pi(z)}{\log z}\ll \dfrac{z}{(\log z)^2},\ \ (z\geqslant 2).$$
	Hence, from \eqref{S2} we deduce that
	$$S_2 \ll 2\uparrow\uparrow(l-1)\dfrac{(x+y)^{{1}/{2\uparrow\uparrow(l-1)}}}{\log(x+y)},$$
	which gives the lemma.

\end{proof}

Now we are ready to prove Theorem \ref{Th6}.
\begin{proof}
For $n$ large enough put
$$l = \max\{\nu\geqslant 2: 2\uparrow\uparrow (\nu-1)\leqslant \sqrt{\log n/\log 2}\}.$$
Let us consider several cases.
\begin{enumerate}
	\item $H(n+1) = H(n+2) = \cdots = H(n+k)<l;$
	\item $H(n+1) = H(n+2) = \cdots = H(n+k)\geqslant l.$
\end{enumerate}
Consider the first case. Among any consecutive $2^{\alpha+1}$ numbers, there is exactly one number that is divisible by $2^\alpha$ and not divisible by $2^{\alpha+1}$. Hence, since $H(2 \uparrow\uparrow l) = l$, it follows that
$$k\leqslant 2\cdot 2\uparrow\uparrow l\leqslant 2^{\sqrt{\log n/\log 2}+1}\leqslant e^{\sqrt{\log n}}.$$
Consider the second case. Lemma \ref{lem6} implies that for some $c_1>0$ we have
\begin{equation}
	k = \sum\limits_{\substack{j=1 \\ H(n+j)\geqslant l}}^k 1\leqslant \dfrac{c_1 k}{2\uparrow\uparrow l} + \dfrac{c_1 2\uparrow\uparrow (l-1)}{\log(n+k)} (n+k)^{{1}/{2\uparrow\uparrow (l-1)}}.
	\end{equation}
	Thus, using Lemma \ref{lem6.2} and the inequality 
	$$2\uparrow\uparrow (l-1) > \frac{1}{2}\log\log n-\frac{1}{2}\log\log 2,$$
	which follows from the definition of $l$, we obtain
	\begin{multline*}
		k\leqslant c_1 \left(1-\dfrac{c_1}{2\uparrow\uparrow l}\right)^{-1}\dfrac{2\uparrow\uparrow(l-1)}{\log n} n^{{1}/{2\uparrow\uparrow(l-1)}}(1+o(1))\\
		\leqslant c_2 \dfrac{2\uparrow\uparrow(l-1)}{\log n} n^{{1}/{2\uparrow\uparrow(l-1)}}\leqslant \dfrac{c_2}{\sqrt{\log 2\log n}} n^{2/\log\log n} \leqslant e^{2\log n/\log\log n}.
		\end{multline*}
	Suppose that $n$ has the form $$n = (2\uparrow\uparrow m)^{2\uparrow\uparrow (m-1)}.$$ Then $l = m$ and
	$2\uparrow\uparrow (m-1) = \sqrt{{\log n}/{\log 2}}.$ Hence,
	$$k\leqslant c_2 \dfrac{2\uparrow\uparrow(l-1)}{\log n} n^{{1}/{2\uparrow\uparrow(l-1)}} \leqslant \dfrac{c_2}{\sqrt{\log 2\log n}} n^{\sqrt{\log 2/\log n}}\leqslant e^{\sqrt{\log n}}.$$
	This completes the proof.
\end{proof}

The following theorem shows that the number of integers $n$ with large $k(n)$ is small.
\begin{theorem}
	Set
	$$K(x;g) = \#\left\{n\leqslant x: k(n)\geqslant g\right\}.$$
	Then uniformly on $x\geqslant 2$ and $g\geqslant 2$ we have $$K(x; g)\ll\dfrac{x}{\log g}.$$
	\end{theorem}
\begin{proof}
	Define $m$ from the inequlity
	$$2\cdot 2\uparrow\uparrow m\leqslant g< 2\cdot 2\uparrow\uparrow (m+1).$$
	Then it easy to see that for each $n\geqslant 1$ there is a $1\leqslant j\leqslant g$ such that $$ n+j \equiv 0 \pmod {2\uparrow\uparrow m}\ \ \text{and}\ \  n+j \not\equiv 0 \pmod {2\cdot2\uparrow\uparrow m}.$$ Thus we get $$H(n+j)\geqslant m.$$ Hence, using \eqref{mainineq} we obtain
	\begin{multline*}
		K(x;g) \leqslant \#\left\{n\leqslant x: t^{\#}(n+1) = \cdots = t^\#(n+g)\right\}\\
		\leqslant \#\left\{n\leqslant x: H(n+1) = \cdots = H(n+g)\right\} \\
		\leqslant \#\left\{n\leqslant x: H(n+1)\geqslant m\right\} \ll \dfrac{x}{2\uparrow\uparrow m}.
		\end{multline*}
	Since $$\dfrac{\log g}{\log 2}-1< 2\uparrow\uparrow m,$$
	we conclude the proof.
\end{proof}
A stronger result than Theorem \ref{Th6} (for numbers not of the type \eqref{type_n}) is the following. 
%We will follow the method from \cite{Spa23}.
\begin{theorem}\label{Spat}
	For $n \to +\infty$  we have
	$$\log k(n) \leq  \left(\dfrac{1}{\sqrt{2}}+o(1)\right)\sqrt{\log n\log\log n}.$$
	\end{theorem}
Note that the theorem follows immediately from the corresponding bound in \cite[Th. 5]{EPS}, where the authors consider the quantity
$$k_{EPS}(n):= \max\left\{k\geqslant 1: \omega(n+1) = \omega(n+2) = \cdots = \omega(n+k)\right\}$$
and prove that
$$k_{EPS}(n)\leqslant \left(\frac{1}{\sqrt{2}}+o(1)\right)\sqrt{\log n\log\log n}.$$
Since trivially we have $k(n)\leqslant k_{EPS}(n),$ the desired result follows. Anyway, it would be interesting to have a better understanding of the relationship between $k(n)$ and $k_{EPS}(n)$.

\begin{theorem}
	For $\omega(n)\to +\infty$ we have
	$$k(n-1)\leqslant \omega(n)^{\omega(n)(1+o(1))}.$$
\end{theorem}
\begin{proof}
	Set
	$$P(n) = \prod_{p\leqslant p_{\omega(n)+1}}p.$$
	For each $n\geqslant 1$ there is a number $j$ such that $0 \leqslant j \leqslant P(n)-1$ and $$n+j \equiv 0 \pmod {P(n)}.$$ Since
	$\omega(n+j)\geqslant \omega(P(n))>\omega(n)$, it follows that 
	$t^\#(n)\neq t^{\#}(n+j).$ Thus 
 $$k(n-1)\leqslant P(n).$$ Using the Prime Number Theorem, we complete the proof.
\end{proof}

It is well known (see \cite[Theorem 010]{HT} and take $t = x$) that
$$\#\left\{n\leqslant x: \omega(n)\leqslant 10\log\log x\right\} = x- o(x).$$
Hence, for almost all $n\leqslant x$ we have
$$\log k(n) = O(\log\log x\log\log\log x),$$
which gives a better bound than the one given by Theorem \ref{Spat}.

%\bigskip 
%A couple of questions for the future:
%what can we say about $\sum_{n \leq x } k(n)$?
%Also, defining $\hat{k}(n) = \max_{m \leq n} k(m)$, what can we say about the asymptotic behaviour of $\hat{k}(n)$?

\section{
When two configurations of trees are equal}
% patterns?

In this section, we study relations between numbers that give the same tree configurations.

\begin{lemma}\label{k_log_nm_nk}
	Let $k$, $m$, and $n$ be positive integers for which
	$m<n$
	and
	\begin{equation}\label{OO}
		O(m, m+1, \ldots, m+k) = O(n, n+1, \ldots, n+k).
		\end{equation}
	Then we have
	$$k\log\dfrac{n+k}{m+k}\ll \sqrt{n+k},$$
% that is there is a positive $c>0$ independent %of $k, m$ and $n$ such that 
 %$$k\log\dfrac{n+k}{m+k}\leq c(n+k)^{\frac{3}%{4}}$$
 %for all $k$, $m$, and $n$ such that $m<n.$
	where the implied constant is absolute.
	\end{lemma}
\begin{proof}
	It follows from \eqref{OO} that for each $0\leqslant a<b\leqslant k$ we have
	\begin{equation}\label{Short_config}
		O(m+a+1, \ldots, m+b) = O(n+a+1, \ldots, n+b).
		\end{equation}
	Set
	$$F(x) = \sum_{n\leqslant x} 2^{\omega(n)}.$$
		%Then using the equalities $$
	%2^{\omega(n)} = \sum\limits_{d|n} %\mu^2(d)
	%$$
	%and
	%$$
	%\sum\limits_{k\leqslant x} \mu^2(k) = %\dfrac{x}{\zeta(2)} + O(\sqrt{x}),
	%$$
	%where $\mu(n)$ denotes the M$\ddot{\text{o}}$bius function, and using the Dirichlet's hyperbola method, one can show that (ADD SOME LINK)
 Hence, from \eqref{Short_config} we find that
	$$F(n+b) - F(n+a) = F(m+b) - F(m+a),\ \ (0\leqslant a<b\leqslant k).$$
	Therefore, the function
	$$f(a) = F(n+a)-F(m+a)$$
	is constant when $0\leqslant a\leqslant k.$
	Choose $a'$, $a$ such that $0\leqslant a'< a\leqslant k$ and
	\begin{equation}\label{a-a'}
		a-a' = k(1-\varepsilon)
		\end{equation}
	for some $0<\varepsilon<1.$  Then using the asymptotic formula (see \cite{Sa} and \cite{GV})
	\begin{equation}\label{asymp}
		F(x) = c_1 x \log x + c_2 x + O\left( \sqrt{x}\right),
	\end{equation}
	where $c_1 = {1}/{\zeta(2)} = {6}/{\pi^2}$, and $c_2$ is some (more complicated) constant, we get 
	\begin{multline}\label{g(a)-g(a')}
		0 = f(a)-f(a') = F(n+a)-F(m+a)-F(n+a')+F(m+a')\\
		%=\left(c_1(n+a)\log(n+a)+c_2(n+a)-c_1(n+a')\log(n+a')-c_2(n+a') \right)\\
		%- \left(c_1(m+a)\log(m+a)+c_2(m+a)-c_1(m+a')\log(m+a')-c_2(m+a') \right)\\
		 %+ O((n+k)^{\frac{3}{4}})\\
		 =c_1((n+a)\log(n+a)-(m+a)\log(m+a))\\
		 -c_1((n+a')\log(n+a')-(m+a')\log(m+a'))+ O(\sqrt{n+k})\\
		 =g(a)-g(a')+ O(\sqrt{n+k}),
	\end{multline}
where
$$g(x) = c_1((n+x)\log(n+x)-(m+x)\log(m+x)).$$
Using the mean value theorem for some $0\leqslant \theta\leqslant 1$ we obtain
\begin{equation}\label{log_n_k_m}
	g(a)-g(a') = g'(a'+\theta(a-a'))(a-a')\geqslant c_1\left( \log\dfrac{n+k}{m+k}\right) (a-a').
\end{equation}
Using \eqref{a-a'}, \eqref{g(a)-g(a')}, and \eqref{log_n_k_m} we conclude the proof.

\end{proof}

Now we prove the following statement. If we have two same  configurations of trees, say, $O(m, m+1, \ldots, m+k)$ and $O(n, n+1, \ldots, n+k)$, and if $k$ is not too large and not too small in terms of $m$ and $n$, then the distance between $n$ and $m$ is small. 

\begin{theorem}
	Let $k$ be large enough, $m \leqslant n$, and $$n^\alpha\leqslant k\leqslant m ^{\beta},$$ where $\alpha $ and $\beta$ are fixed parameters with $1/2< \alpha< \beta< 1.$ Suppose that
	\begin{equation}\label{O} O(m, m+1, \ldots, m+k) = O(n, n+1, \ldots, n+k).
	\end{equation}
	Then we have
	$$n = m + O(m^\gamma),$$
	where $\gamma = 3/2-\alpha<1.$
\end{theorem}
\begin{proof}
	We may assume that $m<n$. It follows from the inequality $n\leqslant k^{1/\alpha}$ and Lemma \ref{k_log_nm_nk} that
	$$\log\dfrac{n+k}{m+k}\ll \dfrac{\sqrt{n+k}}{k}\ll k^{\frac{1}{2\alpha}-1}.$$
	Since ${1}/({2\alpha})-1 < 0$, we get
	$$
	\dfrac{n+k}{m+k} \leqslant e^{O\left( k^{\frac{1}{2\alpha}-1}\right)}= 1+O \left( k^{\frac{1}{2\alpha}-1}\right).
	$$
	Hence, we will have
	\begin{equation}\label{nmk}
	n + k \leqslant m + k + O \left(\dfrac{m}{k^{1-\frac{1}{2\alpha}}} + k^{\frac{1}{2\alpha}} \right).
	\end{equation}
	Since $m^\alpha\leqslant k \leqslant m^\beta$ and the function 
	$$f(k) = \dfrac{m}{k^{1-\frac{1}{2\alpha}}} + k^{\frac{1}{2\alpha}}$$
	 is decreasing for $$0<k\leqslant m\left( 2\alpha-1\right),$$ it follows that the remainder term in \eqref{nmk} does not exceed
	 $$O(f(m^\alpha)) = O(m^{\frac{3}{2}-\alpha}).$$
	 Thus we have
	 $$m<n\leqslant m+ O(m^{\frac{3}{2}-\alpha}).$$
	 The proof of the lemma is complete.
	\end{proof}
Now we show that if we have two same configurations of trees then the length $k$ of these configurations is not too large.
\begin{theorem}\label{Thm7.3}
	Let $k$, $m$, and $n$ be positive integers for which
	$m\leqslant n/2$
	and
	\begin{equation*}
		O(m, m+1, \ldots, m+k) = O(n, n+1, \ldots, n+k).
	\end{equation*}
	Then we have
	$$k \ll \sqrt{n}(\log n)^{\frac{3}{2}},$$
	where the implied constant is absolute.
\end{theorem}
\begin{proof}
	Since $m\neq n$ we have
	$$k < \kappa_+(m)\leqslant cm\log m\leqslant cn\log n$$
	for some sufficiently large $c>0$.

Hence, from Lemma \ref{k_log_nm_nk} we derive
$$k\log\dfrac{n+k}{m+k}\ll \sqrt{n\log n}.$$
Since
\begin{multline*}
	\log \dfrac{n+k}{m+k} = \log \left(1+\dfrac{n-m}{m+k} \right)\\
	\geqslant  \log \left(1+\dfrac{n/2}{2c m\log m} \right)\geqslant \log\left(1+O\left(\dfrac{1}{\log n}\right) \right)\gg \dfrac{1}{\log n}, 
\end{multline*}
we conclude the proof.
\end{proof}

We now prove an analogue of Theorem \ref{Thm7.3} for all $m<n$. Namely, we prove the following
\begin{theorem}\label{MainTh}
Let $k$, $m$, and $n$ be positive integers for which
$m<n$
and
\begin{equation*}
	O(m, m+1, \ldots, m+k) = O(n, n+1, \ldots, n+k).
\end{equation*}
Then we have
$$k\ll n^{\frac{3}{4}}(\log n)^{\frac{3}{2}},$$
where the implied constant is absolute.
\end{theorem}
 Note that Lemma \ref{k_log_nm_nk} is no longer sufficient, since for $m$ close to $n$ it gives too rough bound for $k$.
 
  First, we need one unexpected lemma from the combinatorics on words. Let us fix some notations. Let $\Sigma$ denotes an arbitrary alphabet, and $\Sigma^*$ denotes the set of all finite words over $\Sigma$. For each $\alpha, \beta \in \Sigma^*$ we define $\alpha\beta \in \Sigma^*$ to be the concatenation of $\alpha$ and $\beta$. We also write $$\alpha^n = \underbrace{\alpha\cdots\alpha}_{n \text{ times}}\ \ \ (\alpha\in\Sigma^*,\ \ n\geqslant 1).$$ Let us reserve the symbol $\Lambda$ for the empty word. Thus
 $\alpha\Lambda = \Lambda \alpha$
 for each $\alpha \in \Sigma^*$. Denote the length of a word $\alpha$ by $|\alpha|$. We prove the following 
 \begin{lemma}\label{words}
 	Let $\alpha, \beta, \gamma \in \Sigma^*$ be words such that
 	$\alpha, \gamma \neq \Lambda$ and
 	\begin{equation}\label{ab=bc}
 		\alpha\beta = \beta\gamma.
 		\end{equation}
 	Then we have 
 	$$\alpha\beta\gamma = \alpha^n\delta,$$
 	where $\alpha = \delta\lambda$\, for some $\lambda \in \Sigma^*$ and $\delta \neq \alpha.$
 \end{lemma}
\begin{proof}
	The proof is by induction on $|\beta|.$ For $|\beta| = 0$ we have $\beta = \Lambda$ and $\alpha = \gamma.$ Hence,
	$\alpha\beta\gamma = \alpha^2$ and $\lambda = \delta =\Lambda\neq \alpha.$ 
	
	Assume that the lemma is true for arbitrary $\beta'$ with $|\beta'| = m\geqslant 0.$ Let us prove it for $\beta$ with $|\beta| = m+1.$ We have $\beta = k\beta'$ for some $k\in \Sigma$ and $\beta'\in \Sigma^*$ with $ |\beta'|=m.$ From \eqref{ab=bc} we find that $\alpha k\beta' = k\beta'\gamma$. Since $\alpha \neq \Lambda$, it follows that $\alpha = k\alpha'$ for some $\alpha'\in \Sigma^*.$ Hence, $\alpha'k\beta' = \beta'\gamma.$ According to the inductive hypothesis we get
	$$\alpha'k\beta'\gamma = (\alpha'k)^n\delta,$$
	where $\alpha'k = \delta \lambda$, $n\geqslant 1$, and $\delta, \lambda\in \Sigma^*$. Therefore, we have
	$$\alpha\beta\gamma = k\alpha'k\beta'\gamma = k(\alpha'k)^n\delta = (k\alpha')^nk\delta = \alpha^n k\delta.$$
	 Consider three cases:
	 \begin{enumerate}
	 	\item $\delta = \Lambda;$
	 	\item $\delta \neq \Lambda,\ \ \lambda = \Lambda;$
	 	\item $\delta\neq \Lambda,\ \ \lambda\neq\Lambda.$
	 \end{enumerate}
 
 In the first case we have $\alpha\beta\gamma = \alpha^nk.$ Hence, if $\alpha' = \Lambda$, then $\alpha\beta\gamma = k^{n+1} = \alpha^{n+1}$. Else if $\alpha'\neq\Lambda$, then $\alpha\beta\gamma = \alpha^n k$ and $k\neq \alpha.$
 
 In the second case we have $\alpha'k = \delta.$ Since $\delta \neq \Lambda$, it follows that $\delta = \delta'k$ for some $\delta'.$ Hence, $\alpha' = \delta'$ and
 $k\delta = k\delta'k = k\alpha'k = \alpha k.$ Thus we get
 $$\alpha\beta\gamma = \alpha^n k\delta = \alpha^{n+1}k.$$ If $\alpha' = \Lambda$, then $\alpha = k$ and $\alpha\beta\gamma = \alpha^{n+2}.$ Else if $\alpha'\neq \Lambda,$ then $\alpha\beta\gamma = \alpha^{n+1}k$ and $k\neq \alpha.$
 
 In the third case for $\delta,\lambda\neq\Lambda$ we have
 $$\alpha'k = \delta\lambda.$$ Hence, $\lambda = \lambda'k$ for some $\lambda'$, and $\alpha' = \delta\lambda'.$ Thus $\alpha = k\alpha' = k\delta\lambda'$. If $\lambda' = \Lambda$, then $\alpha\beta\gamma = \alpha^nk\delta = \alpha^{n+1}$. Else if $\lambda'\neq \Lambda$, then $\alpha\beta\gamma = \alpha^n k\delta$ and $k\delta\neq \alpha.$
	The proof of the lemma is complete.
\end{proof}
The next lemma gives a bound for the lowest prime which not divided a given number.

\begin{lemma}\label{prime_no_div}
	Let $q\geqslant 2$ and
	\begin{equation}\label{rhoq}
		\rho(q) = \min\left\{p\ \  \text{prime}: p\nmid q\right\}.
	\end{equation}
	Then we have
	$$\rho(q)\ll \log q.$$
\end{lemma}
\begin{proof}
	The lemma is trivial for odd $q$. So, we may assume that $q$ is even. Hence, the canonical decomposition of $q$ has the form
	$$q = p_1^{\alpha_1}\cdots p_i^{\alpha_i}q_{i+1}^{\alpha_{i+1}}\cdots q_s^{\alpha_s}\ \ (i\geqslant 1),$$
	where $s = \omega(q)$, $p_j$ denotes the $j^{\text{th}}$ prime, $\alpha_j\geqslant 1$ for $1\leqslant j\leqslant s$, and $$p_{i+1}<q_{i+1}<\cdots<q_s.$$
	Then on the one hand, by Bertrand's postulate (see \cite{Mur08} Ch. 3, Th. 3.1.9) we have
	$\rho(q) = p_{i+1} \leqslant 2p_i.$ On the other hand, by Chebyshev's bound we have
	$$q\geqslant p_1p_2\cdots p_i \geqslant e^{c p_i}\ \ (c>0).$$
	 Thus $\rho(q)\leqslant (2/c)\log q.$ The proof of the lemma is complete.	 \end{proof}
 	 
 	 Now we are going to prove Theorem \ref{MainTh}.
	 \begin{proof}
	 	Consider two cases:
	 	\begin{enumerate}
	 		\item $m\leqslant n-n^\alpha,$
	 		\item $n-n^\alpha<m<n,$
	 		\end{enumerate}
 		where $0<\alpha<1$ will be defined later. In the first case, since $$k<\kappa_+(m)\leqslant c m\log m$$ for some $c>0$, we have
 		\begin{multline*}
 			\log\dfrac{n+k}{m+k} = \log\left(1+\dfrac{n-m}{m+k}\right)\geqslant \log\left(1+\dfrac{n-m}{c m\log m}\right)\\
 			\geqslant \log\left(1+O\left( \dfrac{n^\alpha}{n\log n}\right) \right)\gg \dfrac{1}{n^{1-\alpha}\log n}.
 		\end{multline*}
 	Since 
 	$$k<\kappa_+(n)\ll n\log n,$$
 	it follows from Lemma \ref{k_log_nm_nk} that
 	\begin{equation}\label{ineq(-1)}
 		k\ll \dfrac{\sqrt{n+k}}{\log\frac{n+k}{m+k}} \ll n^{\frac{3}{2}-\alpha}(\log n)^{\frac{3}{2}}.
 		\end{equation}
 	
 	Consider the second case. If
 	\begin{equation}\label{ineq0}
 		k\leqslant n^{\frac{1}{20}},
 		\end{equation}
 	then we are done. So, we may assume that $k>n^{1/20}.$ If
 	 $m+k\leqslant n,$ then 
 	\begin{equation}\label{ineq1}
 	k<n^\alpha.
 	\end{equation}
 	Else we have $k>q:=n-m.$ If $k/q< k^\varepsilon$ for some $\varepsilon = \varepsilon(n)>0$, then
 	\begin{equation}\label{ineq2}
 		k<q^{\frac{1}{1-\varepsilon}}\leqslant n^{\frac{\alpha}{1-\varepsilon}}.
 	\end{equation}
 Suppose that $k/q>k^{\varepsilon}.$ Define the value $M$ from the inequalities
 \begin{equation*}
 	P:= \prod\limits_{\substack{p\leqslant M \\ p \nmid q}}p\leqslant \dfrac{k}{q} < \prod\limits_{\substack{p\leqslant M+1 \\ p \nmid q}}p.
 \end{equation*}
Then we obtain
\begin{equation*}
	\dfrac{k}{q}\leqslant (M+1)^{\pi(M+1)-\omega_{M+1}(q)}\leqslant (M+1)^{\pi(M)-\omega_M(q)+1},
\end{equation*}
 		where $$\omega_x(q):= \sum\limits_{\substack{p|q \\ p\leqslant x}} 1.$$
 		Hence,
 		\begin{equation}\label{piM}
 			\pi(M)-\omega_M(q)+1\geqslant \dfrac{\log\frac{k}{q}}{\log(M+1)}\geqslant \dfrac{\varepsilon \log k}{\log(M+1)}.
 			\end{equation}
 		On the other hand, we have
 		$$\prod_{p\leqslant M}p\leqslant \dfrac{k}{q}\prod_{p|q}p\leqslant k.$$
 		Hence, $M+1 \leqslant \log k(1+o(1))$. In what follows, we will choose $\varepsilon$ so that
 		\begin{equation}\label{cond_eps}
 			1 = o\left(\dfrac{\varepsilon\log k}{(\log\log k)^2} \right) 
 		\end{equation}
 	and
 	\begin{equation}\label{cond_eps2}
 		k^{1-\varepsilon} = o(k).
 	\end{equation}
 	Hence, from \eqref{piM} we find that
 	\begin{multline*}\pi(M)-\omega_M(q)\geqslant \dfrac{\varepsilon \log k}{\log\log k+o(1)}-1 \\
 		= \dfrac{\varepsilon \log k}{\log\log k} + o\left( \dfrac{\varepsilon \log k}{(\log\log k)^2}\right)= \dfrac{\varepsilon \log k}{\log\log k}(1+o(1)).
 		\end{multline*}
 	Note that $k = o(n).$ Indeed, due to Lemma \ref{words} for each $a\in\left[m,n\right)$ we have
 	\begin{equation}\label{trees}
 		t(a+q) = t(a+2q) = \cdots = t(a +\eta q),
 		\end{equation}
 	where $\eta = \lfloor k/q \rfloor.$ If $\eta\leqslant 2\rho(q)$, where $\rho(q)$ is defined in \eqref{rhoq}, then using Lemma \ref{prime_no_div} we obtain $\eta\ll \log q$ and
 	\begin{equation}\label{ineq3}
 		k\ll q\log q\ll n^\alpha \log n = o(n).
 	\end{equation}
 Otherwise, if $\eta>2\rho(q)$, then there exists $1\leqslant j\leqslant 2\rho<\eta$ such that
 $a+jq \equiv 0 \pmod \rho$ and $a+jq$ is not a prime. Hence, 
 $$a+q, a+2q, \ldots, a+\eta q$$ are composite numbers for all $a\in\left[m,n\right).$ Therefore, there are no primes in the set
 $$\bigsqcup_{a\in\left[m,n\right) }\left\{a+q, a+2q, \ldots, a+\eta q \right\} = \left[n, n+ \eta q\right).$$
 Hence, $\eta q = o(n)$ and $k = o(n)$, as desired. It follows from \eqref{trees} that
 $$E(a+q) = E(a+2q) = \cdots = E(a+\eta q),\ \  \forall a\in \left[m,n\right).$$
 Since $(P,q) = 1$, there is some $j$ such that $1\leqslant j \leqslant P\leqslant k/q$ and $$a+jq \equiv 0 \pmod P.$$ Thus we get that
 $$E(a+jq)\geqslant \omega(P) = \pi(M)-\omega_M(q)\geqslant \dfrac{\varepsilon \log k}{\log\log k}(1+o(1)).$$
Hence, using \eqref{trees} we obtain
 $$ \sum_{j=1}^\eta E(a+jq)\geqslant \dfrac{\eta\varepsilon \log k}{\log\log k}(1+o(1)).$$
 Summing this inequality over all $a\in \left[m,n\right)$, we get
 $$\eta q \dfrac{ \varepsilon \log k}{\log\log k}(1+o(1))\leqslant \sum_{a = m}^{n-1}\sum_{j=1}^\eta E(a+jq)\leqslant \sum_{j=0}^{k}E(n+j).$$
 It follows from \eqref{cond_eps2} that $k - q\eta < q\leqslant k^{1-\varepsilon} = o(k)$. Thus we conclude that
 $$\dfrac{ k \varepsilon \log k}{\log\log k}(1+o(1))\leqslant \sum_{j=0}^{k}E(n+j).$$ Now we use some ideas from \cite{Spa23}.
 Let us find an upper bound for the last sum. We have
 $$\sum_{j=0}^k E(n+j) = \sum_{j=0}^k\sum_{p|n+j}(1+E(\nu_p(n+j))) = \mathcal{E}_1+\mathcal{E}_2,$$
 	 where $\mathcal{E}_1$ denotes the contribution of $p\leqslant k$, and $\mathcal{E}_2$ denotes the contribution of the remaining $p$. For $\mathcal{E}_1$ we have
 	\begin{equation*}
 		\mathcal{E}_1 = \sum_{p \leqslant k}\sum_{\nu \leqslant \frac{\log(n+k)}{\log p}}(1+E(\nu))\sum\limits_{\substack{j=0 \\ p^\nu || n+j}}^k 1 \leqslant \sum_{p \leqslant k}\sum_{\nu \leqslant \frac{\log(n+k)}{\log p}}(1+E(\nu))\left(\dfrac{k+1}{p^\nu}+1\right), 
 	\end{equation*}
 	where $p^\nu || n$ means that $p^\nu | n$ and $p^{\nu+1}  \nmid n.$
 	Since $E(\nu)\leqslant \nu-1$, we see that the contribution from the term $(k+1)/p^\nu$ is just
 	$$(k+1) \sum_{p\leqslant k}\sum_{\nu \leqslant \frac{\log(n+k)}{\log p}}\dfrac{1+E(\nu)}{p^\nu} \ll k \sum_{p\leqslant k} \dfrac{1}{p} \ll k \log\log k.$$
 	Using the inequality (see \cite{Iud21}, Theorem 1) 
 	$$\sum_{\nu\leqslant x}E(\nu)\ll x \log\log x,\ \ (x\geqslant 2),$$
 	we get that the contribution from $1$ is bounded by
 	\begin{multline*}
 		\sum_{p \leqslant k}\sum_{\nu \leqslant \frac{\log(n+k)}{\log p}}(1+E(\nu))\ll  \log(n+k)\sum_{p\leqslant k}\dfrac{1}{\log p} + \sum_{p \leqslant k}\dfrac{\log(n+k)}{\log p}\log\log\log(n+k)\\
 		\ll \dfrac{k}{(\log k)^2}\log(n+k)\log\log\log(n+k).
 	\end{multline*}
 	Since $k = o(n)$, we obtain
 	$$\mathcal{E}_1 \ll \dfrac{k}{(\log k)^2}\log n\log\log\log n + k \log\log k.$$
 	
 	For $\mathcal{E}_2$ we have
 	\begin{equation*}
 		\mathcal{E}_2 = \sum_{j=0}^k\sum\limits_{\substack{p| n+j \\ p>k}}\left(1+E(\nu_p(n+j)) \right)\leqslant  \sum_{j=0}^k\sum\limits_{\substack{p| n+j \\ p>k}}\nu_p(n+j).
 	\end{equation*}
 	Since
 	$$n+j = \prod_{p|n+j} p^{\nu_p(n+j)} > \prod\limits_{\substack{p| n+j \\ p>k}} k^{\nu_p(n+j)},\ \  (0\leqslant j\leqslant k),$$
 	it follows that
 	$$\sum\limits_{\substack{p| n+j \\ p>k}}\nu_p(n+j) \leqslant \dfrac{\log(n+k)}{\log k}.$$
 	Hence, 
 	$$\mathcal{E}_2\leqslant \dfrac{k \log n}{\log k}(1+o(1)).$$
 	Thus we obtain
 	\begin{multline*}\dfrac{ k \varepsilon \log k}{\log\log k}(1+o(1)) \\
 		\leqslant \dfrac{k \log n}{\log k}(1+o(1)) + O(k \log\log k) + O\left( \dfrac{k}{(\log k)^2}\log n\log\log\log n\right) .
 		\end{multline*}
  Since
  $$\log\log k = o\left(\dfrac{\varepsilon \log k}{\log\log k}\right)$$
  and
  $$\dfrac{\log\log\log n}{\log k}=o(1),$$
  we get that
 	$$\dfrac{\varepsilon\log k}{\log\log k}\leqslant \dfrac{\log n}{\log k}(1+o(1)).$$
 	That is for some $2\leqslant L \leqslant (1/\varepsilon)\log n (1+o(1))$ we have
 	$$\dfrac{(\log k)^2}{\log\log k} = L.$$ Hence,
 	$$2\log\log k(1+o(1)) = \log L$$
 	and
 	\begin{equation}\label{ineq4}
 		\log k \leqslant \sqrt{\log n^{\frac{1}{\varepsilon}}\log\log n^{\frac{1}{\varepsilon}}} \left(\frac{1}{\sqrt{2}}+o(1)\right).
 		\end{equation}
 Choose $\alpha = 3/4$ and
 $\varepsilon = {2\log\log n}/{\log n}.$ Then it follows from \eqref{ineq(-1)}, \eqref{ineq0}, \eqref{ineq1}, \eqref{ineq2}, \eqref{ineq3}, and \eqref{ineq4} that
 $$k \ll n^{\frac{3}{4}}(\log n)^{\frac{3}{2}}.$$
 This completes the proof of Theorem \ref{MainTh}.
	
\end{proof}

\noindent{\bf Acknowledgments}
R.C. is partially supported by Sapienza Universit\`a di Roma (Progetti di Ateneo 2019, 2020). P.C. is partially supported by Alma Mater Studiorum University of Bologna and by PRIN project 2022 N.2022B5LF52. The work of V.I. was supported by the Theoretical Physics and Mathematics Advancement Foundation “BASIS”. The work of V.I. is also supported in part by the Moebius Contest Foundation for Young Scientists.
%V.I. is a winner of the all-Russia mathematical August Moebius contest of graduate and undergraduate student papers and thanks the jury and the board for the high praise of his work.

% \section{Preliminaries}

%Here we introduce trees, observables, milestones, and $\kappa_+$ in a precise way.

% To each natural number $n$, we associate a rooted planar tree $t(n)$ as follows .....  Sometimes we will also consider the corresponding nonplanar tree $t^\#(n)$. For instance, $t(12) \neq t(18)$, but $t^\#(12) = t^\#(18)$.

% An observable is a finite pattern of trees. 

%Given a natural number $n \geq 1$, we define $\kappa_+(n) \geq 1$ to be the least positive integer such that $O(n,n+1,\ldots,n+\kappa_+(n))$ is unique.
%In [CC] it has been shown that the function $\kappa_+: {\mathbb N} \to {\mathbb N}$ is well-defined, i.e. $\kappa_+(n)$ exists, for all $n$.

% GENERALIZE $\kappa_+$

\end{document}